\documentclass[11pt,a4paper]{amsart}
\usepackage{amsmath, amsthm, amsfonts, amssymb}
\usepackage[bookmarks=false]{hyperref}

\newtheorem{letterthm}{Theorem}

\newtheorem{lettercor}[letterthm]{Corollary}

\newtheorem{thm}{Theorem}[section]
\newtheorem{lem}[thm]{Lemma}

\newtheorem{prop}[thm]{Proposition}

\theoremstyle{definition}
\newtheorem{rem}[thm]{Remark}

\newtheorem{nota}[thm]{Notation}
\newtheorem{df}[thm]{Definition}

\newtheorem*{claim}{Claim}

\newcommand{\R}{\mathbf{R}}
\newcommand{\C}{\mathbf{C}}

\newcommand{\N}{\mathbf{N}}

\newcommand{\Ad}{\operatorname{Ad}}
\newcommand{\id}{\text{\rm id}}

\newcommand{\rL}{\mathord{\text{\rm L}}}
\newcommand{\rE}{\mathord{\text{\rm E}}}
\newcommand{\Sp}{\mathord{\text{\rm Sp}}}

\newcommand{\Ball}{\mathord{\text{\rm Ball}}}

\newcommand{\dpr}{^{\prime\prime}}

\begin{document}

\title[Free independence in ultraproduct von Neumann algebras and applications]{Free independence in ultraproduct \\ 
von Neumann algebras and applications}

\begin{abstract}
We prove an analogue of Popa's free independence result for subalgebras of ultraproduct ${\rm II_1}$ factors \cite{Po95b} in the framework of ultraproduct von Neumann algebras $(M^\omega, \varphi^\omega)$ where $(M, \varphi)$ is a $\sigma$-finite von Neumann algebra endowed with a faithful normal state satisfying $(M^{\varphi})' \cap M = \C 1$. More precisely, we show that whenever $P_1, P_2 \subset M^\omega$ are von Neumann subalgebras with separable predual that are globally invariant under the modular automorphism group $(\sigma_t^{\varphi^\omega})$, there exists a unitary $v \in \mathcal U((M^\omega)^{\varphi^\omega})$ such that $P_1$ and $v P_2 v^*$ are $\ast$-free inside $M^\omega$ with respect to the ultraproduct state $\varphi^\omega$. Combining our main result with the recent work of Ando-Haagerup-Winsl\o w \cite{AHW13}, we obtain a new and direct proof, without relying on Connes-Tomita-Takesaki modular theory, that Kirchberg's quotient weak expectation property (QWEP) for von Neumann algebras is stable under free product. Finally, we obtain a new class of inclusions of von Neumann algebras with the relative Dixmier property. 
\end{abstract}

\author{Cyril Houdayer}
\address{CNRS - Universit\'e Paris-Est - Marne-la-Vall\'ee \\
    LAMA UMR 8050 \\ 77454 Marne-la-Vall\'ee cedex~2 
\\ France}
\email{cyril.houdayer@u-pem.fr}
\thanks{CH is supported by ANR grant NEUMANN and JSPS Invitation Fellowship Program for Research in Japan FY 2014.}

\author{Yusuke Isono}
\address{RIMS, Kyoto University, 606-8502 Kyoto, Japan}
\email{isono@kurims.kyoto-u.ac.jp}
\thanks{YI is supported by JSPS Research Fellowship}



\maketitle

\section{Introduction and statement of the main results}

In his seminal article \cite{Po95b}, Popa obtained a free independence result for subalgebras of ultraproduct ${\rm II_1}$ factors. Among other results, he showed that whenever $P_1, P_2 \subset M^\omega$ are von Neumann subalgebras with separable predual of an ultraproduct ${\rm II_1}$ factor $(M^\omega, \tau^\omega) = (M_n, \tau_n)^\omega$, there exists a unitary $v \in \mathcal U(M^\omega)$ such that $P_1$ and $v P_2 v^*$ are $\ast$-free inside $M^\omega$ with respect to the ultraproduct trace $\tau^\omega$. This result had important consequences regarding Connes's approximate embedding problem \cite{Co75, Ki92, Oz03}. Indeed, it implied that the class of tracial von Neumann algebras with separable predual that are embeddable into $R^\omega$ in a trace-preserving way is stable under free product. Very recently, Popa further investigated the independence properties of subalgebras of ultraproduct ${\rm II_1}$ factors. Firstly, he obtained in \cite{Po13a} the first class of maximal abelian subalgebras of ultraproduct ${\rm II_1}$ factors satisfying the Kadison-Singer conjecture and secondly he generalized in \cite{Po13b} his earlier results from \cite{Po95b} to the case of tracial amalgamated free product von Neumann algebras.

In this paper, we prove an analogue of Popa's free independence result for subalgebras of ultraproduct ${\rm II_1}$ factors \cite{Po95b} in the framework of ultraproduct von Neumann algebras $(M^\omega, \varphi^\omega)$ where $(M, \varphi)$ is a $\sigma$-finite von Neumann algebra endowed with a faithful normal state satisfying $(M^{\varphi})' \cap M = \C 1$. We refer to \cite{AH12, Oc85} and Section \ref{preliminaries} for the construction of the ultraproduct von Neumann algebra $(M^\omega, \varphi^\omega)$.

\begin{letterthm}\label{thmA}
Let $(M, \varphi)$ be any non-type {\rm I} $\sigma$-finite factor endowed with a faithful normal state and $Q \subset M^\varphi$ any von Neumann subalgebra satisfying $Q' \cap M = \C 1$. Let $\omega \in \beta(\N) \setminus \N$ be any non-principal ultrafilter. For all $i \in \{1, 2\}$, let $P_i \subset M^\omega$ be a von Neumann subalgebra with separable predual that is globally invariant under the modular automorphism group $(\sigma_t^{\varphi^\omega})$.

Then there exists a unitary $v \in \mathcal U(Q^\omega)$ such that $P_1$ and $v P_2 v^*$ are $\ast$-free inside $M^\omega$ with respect to the ultraproduct state $\varphi^\omega$.
\end{letterthm}

The proof of Theorem \ref{thmA} uses Popa's {\em incremental patching method} from \cite{Po95b, Po13a, Po13b} with some notable differences in its technical implementation. It can be divided into two main steps. We refer to Section \ref{incremental} for more details.

The {\bf first step of the proof} (see Lemma \ref{technical-lemma1}) consists in proving an ``approximate version" of Theorem \ref{thmA} for the von Neumann algebra $(M, \varphi)$. This result (as well as its proof) is an adaptation of \cite[Lemma 1.2]{Po95b} and is the key result to prove Theorem \ref{thmA}. Like in \cite[Lemma 1.2]{Po95b}, the proof of Lemma \ref{technical-lemma1} uses a maximality argument and relies in a crucial way on Popa's {\em local quantization principle} (see  \cite{Po92, Po95a} and Proposition \ref{local-quantization-principle}). Due to the absence of a trace on the von Neumann algebra $(M, \varphi)$, we work with elements in $M$ that are analytic with respect to the modular automorphism group $(\sigma_t^\varphi)$ and use techniques from \cite[Section~4]{AH12}.

The {\bf second step of the proof} (see Lemmas \ref{technical-lemma2}, \ref{technical-lemma3}, \ref{technical-lemma4}) consists in using the first step of the proof and the modularity assumption on the von Neumann subalgebras $P_1, P_2 \subset M^\omega$ in order to build {\em via} several diagonalization and patching procedures  a unitary element $v \in \mathcal U((M^\varphi)^\omega)$ that satisfies the conclusion of Theorem \ref{thmA}. We observe that unlike \cite{Po95b}, we need to assume that $M_n = M$ is a constant sequence of von Neumann algebras in order to have $M \subset M^\omega$.

Let $M$ and $\mathcal N$ be any von Neumann algebras. We will say that $M$ {\em embeds with expectation into} $\mathcal N$ if there exist a normal $\ast$-embedding $\pi : M \to \mathcal N$ and a faithful normal conditional expectation $\rE_{\pi(M)} : \mathcal N \to \pi(M)$. Using Theorem \ref{thmA} and \cite[Theorem 4.20]{AH12}, we show that the class of von Neumann algebras with separable predual that embed with expectation into a $\sigma$-finite factor of type ${\rm III_1}$ endowed with a faithful normal state that has a large centralizer is stable under free product.

\begin{lettercor}\label{corB}
Let $(M, \theta)$ be any $\sigma$-finite factor of type ${\rm III_1}$ endowed with a faithful normal state such that $(M^\theta)' \cap M = \C 1$.  For all $i \in \{1, 2\}$, let $M_i$ be any von Neumann algebra with separable predual that embeds with expectation into $M^\omega$. For all $i \in \{1, 2\}$, let $\varphi_i \in (M_i)_\ast$ be any faithful normal state. Then the free product von Neumann algebra $(M_1, \varphi_1) \ast (M_2, \varphi_2)$ embeds with expectation into~$M^\omega$.
\end{lettercor}

Very recently, Ando-Haagerup-Winsl\o w showed in \cite[Theorem 1.1]{AHW13} that a von Neumann algebra $M$ with separable predual satisfies Kirchberg's quotient weak expectation property (QWEP) \cite{Ki92} if and only if $M$ embeds with expectation into $R_\infty^\omega$, where $R_\infty$ denotes the unique hyperfinite factor of type ${\rm III_1}$. Hence, a combination of their result and  Corollary \ref{corB} shows that the class of QWEP von Neumann algebras with separable predual is stable under free product. We point out that this last result can also be proven using Connes-Tomita-Takesaki modular theory and Brown-Dykema-Jung's result \cite[Corollary 4.5]{BDJ06}. However, we believe that it is interesting to have a new and direct proof that the QWEP for von Neumann algebras is stable under free product, without relying on Connes-Tomita-Takesaki modular theory.

It also follows from the above discussion that any free product of {\em amenable} von Neumann algebras with respect to arbitrary faithful normal states has the QWEP. Observe that Shlyakhtenko's free Araki-Woods factors $\Gamma(H_\R, U_t)\dpr$ associated with almost periodic orthogonal representations $U : \R \to \mathcal O(H_\R)$ are $\ast$-isomorphic to free products of type ${\rm I}$ von Neumann algebras \cite{Sh96}. Thus, Corollary \ref{corB} gives an alternate proof of QWEP for the almost periodic free Araki-Woods factors \cite{PS02}. We refer to \cite{No05} for more general results on QWEP for $q$-Araki-Woods von Neumann algebras.

An inclusion of von Neumann algebras $\mathcal N \subset \mathcal M$ with faithful normal conditional expectation $\rE_{\mathcal N} : \mathcal M \to \mathcal N$ is said to have the {\em relative Dixmier property} if for every $x \in \mathcal M$, we have
$$\overline{\rm co}^{\|\cdot\|_\infty}\{ u x u^* : u \in \mathcal U( \mathcal N)\} \cap \mathcal N' \cap \mathcal M \neq \emptyset.$$
Here, $\overline{\rm co}^{\|\cdot\|_\infty}\{ u x u^* : u \in \mathcal U(\mathcal N)\}$ denotes the closure with respect to the uniform norm $\|\cdot\|_\infty$ of the convex hull of the uniformly bounded set $\{ u x u^* : u \in \mathcal U(\mathcal  N)\}$. Haagerup showed in \cite[Theorem 3.1]{Ha88} that whenever $M$ is a $\sigma$-finite factor of type ${\rm III_\lambda}$ with $0 < \lambda < 1$ and $\varphi \in M_\ast$ is a faithful normal $\frac{2 \pi}{|\log(\lambda)|}$-periodic state, the inclusion $M^\varphi \subset M$ has the relative Dixmier property. Popa showed in \cite[Corollary 2.4]{Po95b} that whenever $M$ is a ${\rm II_1}$ factor and $Q \subset M$ is an irreducible subfactor, the inclusion $Q^\omega \subset M^\omega$ has the relative Dixmier property. We refer to \cite{Po97, Po13a} for other results regarding the relative Dixmier property.

Following the strategy of \cite{Po95b} and using Theorem \ref{thmA}, we obtain a new class of inclusions of von Neumann algebras $\mathcal N \subset \mathcal M$ with the relative Dixmier property. 

\begin{letterthm}\label{thmC}
Let $(M, \varphi)$ be any non-type {\rm I} $\sigma$-finite factor  endowed with a faithful normal state and $Q \subset M^\varphi$ any von Neumann subalgebra satisfying $Q' \cap M = \C 1$. Let $\omega \in \beta(\N) \setminus \N$ be any non-principal ultrafilter. Then the inclusion $Q^\omega \subset M^\omega$ satisfies the {\em relative Dixmier property}, that is, for all $x \in M^\omega$, we have
$$\overline{\rm co}^{\|\cdot\|_\infty}\{ u x u^* : u \in \mathcal U(Q^\omega)\} \cap \C 1 = \{ \varphi^\omega(x)1\}.$$
\end{letterthm}

Theorem \ref{thmC} moreover implies that there exists a unique (not necessarily normal) conditional expectation $\rE : M^\omega \to Q^\omega$. It is the unique $\varphi^\omega$-preserving conditional expectation $\rE_{Q^\omega} : M^\omega \to Q^\omega$. Let now $M = R_\infty$ be the unique hyperfinite factor of type ${\rm III_1}$ and choose a faithful normal state $\varphi \in M_\ast$ such that  $(M^\varphi)' \cap M = \C 1$. By Theorem \ref{thmC}, the inclusion $(M^\varphi)^\omega \subset M^\omega$ has the relative Dixmier property.

\subsection*{Acknowledgments}

This work was done when the first named author was visiting the Research Institute for Mathematical Sciences (RIMS) in Kyoto during Summer 2014. He gratefully acknowledges the kind hospitality of RIMS. The authors thank Narutaka Ozawa for insightful discussions about the QWEP conjecture and Sorin Popa for useful remarks. They finally thank the anonymous referee for providing valuable comments.

\section{Preliminaries}\label{preliminaries}

\subsection*{Background on $\sigma$-finite von Neumann algebras}

Let $M$ be any $\sigma$-finite von Neumann algebra. We denote by $\Ball(M)$ the unit ball of $M$ with respect to the uniform norm $\|\cdot\|_\infty$, $\mathcal U(M)$ the group of unitaries in $M$ and $\mathcal Z(M)$ the center of $M$. Let $\varphi \in M_\ast$ be any faithful normal state. For all $x \in M$, write $\|x\|_\varphi = \varphi(x^*x)^{1/2}$. Recall that the strong topology on uniformly bounded subsets of $M$ coincides with the topology defined by $\|\cdot\|_\varphi$.

For every $f \in \rL^1(\R)$, we define the {\em Fourier transform} of $f$ by 
$$\widehat f(\xi) = \int_\R f(t) \exp({\rm i t \xi}) \, {\rm d}t, \forall \xi \in \widehat \R = \R.$$
For every $f \in \rL^1(\R)$ and every $x \in M$, put
$$\sigma_f^\varphi (x) = \int_\R f(t) \sigma_t^\varphi(x) \, {\rm d}t.$$
For every $x \in M$, put
$$\Sp_{\sigma^\varphi}(x) = \left\{ \xi \in \widehat \R : \widehat f(\xi) = 0 \text{ for all } f \in \rL^1(\R) \text{ such that } \sigma_f^{\varphi}(x) = 0 \right\}.$$
For every subset $E \subset \widehat \R$, the {\em spectral subspace} of $\sigma^\varphi$ in $M$ associated with $E$ is defined by
$$M(\sigma^\varphi, E) = \left\{ x \in M : \Sp_{\sigma^\varphi}(x) \subset E \right \}.$$
The spectral subspaces satisfiy the following well-known properties (see e.g.\ \cite[Corollary XI.1.8]{Ta03}): for all subsets $E, F \subset \widehat \R$, we have
\begin{itemize}
\item $M(\sigma^\varphi, E)^* = M(\sigma^{\varphi}, -E)$,
\item $M(\sigma^{\varphi}, E) M(\sigma^\varphi, F) \subset M(\sigma^\varphi, \overline{E + F})$ and 
\item $M(\sigma^\varphi, \{0\}) = M^\varphi$.
\end{itemize}

The next proposition will be very useful to prove the main technical results of Section \ref{incremental}.

\begin{prop}[{\cite[Lemma 4.13]{AH12}}]\label{analytic}
Let $(M, \varphi)$ be any $\sigma$-finite von Neumann algebra endowed with a faithful normal state. Let $a > 0$. Any element $x \in M(\sigma^\varphi, [-a, a])$ is analytic with respect to the modular automorphism group $(\sigma_t^\varphi)$ and satisfies
$$\|\sigma_z^\varphi(x)\|_\infty \leq C(a, z) \|x\|_\infty, \forall z \in \C$$
where $C(a, z) = 2 \exp(2a |\Im(z)|) + \exp(a |\Im(z)|)$.
\end{prop}

We also recall the following well-known result.

\begin{thm}[{\cite[Theorem 4.2]{Ta03}}]
Let $(M, \varphi)$ be any $\sigma$-finite von Neumann algebra endowed with a faithful normal state and $N \subset M$ any von Neumann subalgebra. The following conditions are equivalent:
\begin{enumerate}
\item There exists a (unique) $\varphi$-preserving conditional expectation $\rE_N : M \to N$.
\item The von Neumann subalgebra $N$ is globally invariant under the modular automorphism group $(\sigma_t^\varphi)$.
\end{enumerate}
\end{thm}

\subsection*{Free product von Neumann algebras}

Let $I$ be any non-empty set. For all $i \in I$, let $(M_i, \varphi_i)$ be any $\sigma$-finite von Neumann algebra endowed with a faithful normal state. The {\em free product von Neumann algebra} $(M, \varphi) = \ast_{i \in I} (M_i, \varphi_i)$ is the unique (up to state-preserving isomorphism) von Neumann algebra $M$ generated by $M_i$ with $i \in I$ and where the faithful normal state $\varphi$ satisfies the following $\ast$-{\em freeness condition}:
\begin{equation}\label{freeness}
\varphi(x_1 \cdots x_n) = 0 \; \text{ whenever } \; x_j \in M_{i_j} \ominus \C, i_1, \dots, i_n \in I \; \text{ and } \; i_1 \neq \cdots \neq  i_{n}.
\end{equation}
We will say that the von Neumann subalgebras $M_i$ are $\ast$-{\em free inside} $M$ {\em with respect to} $\varphi$ if Equation $(\ref{freeness})$ is satisfied. Here and in what follows, we denote by $M_i \ominus \C 1 = \ker(\varphi_i)$. We refer to the product $x_1 \cdots x_n$ where $x_j \in M_{i_j} \ominus \C 1$, $i_1, \dots, i_n \in I$ and $i_1 \neq \cdots \neq i_{n}$ as a {\em reduced word} in $(M_{i_1} \ominus \C 1) \cdots (M_{i_n} \ominus \C 1)$ of {\em length} $n \geq 1$. The linear span of $1$ and of all the reduced words in $(M_{i_1} \ominus \C 1) \cdots (M_{i_n} \ominus \C 1)$ where $n \geq 1$, $i_1, \dots, i_n \in I$ and $i_1 \neq \cdots \neq i_{n}$ forms a unital $\sigma$-strongly dense $\ast$-subalgebra of $M$.

For all $t \in \R$, we have $\sigma_t^\varphi = \ast_{i \in I} \sigma_t^{\varphi_i}$ (see \cite[Lemma 1]{Ba93} and \cite[Theorem 1]{Dy92}). By \cite[Theorem IX.4.2]{Ta03}, there exists a unique $\varphi$-preserving faithful normal conditional expectation $\rE_{M_i} : M \to M_i$. Moreover, we have $\rE_{M_i}(x_1 \cdots x_n) = 0$ for all the reduced words $x_1 \cdots x_n$ which contains at least one letter from $M_j \ominus \C 1$ with $j \neq i$ (see \cite[Lemma 2.1]{Ue11}). For more on free product von Neumann algebras, we refer the reader to \cite{Ue98, Ue11, Vo85, Vo92}.

\subsection*{Ultraproduct von Neumann algebras}

Let $M$ be any $\sigma$-finite von Neumann algebra. Define
\begin{align*}
\mathcal I_\omega(M) &= \left\{ (x_n)_n \in \ell^\infty(\N, M) : x_n \to 0 \ast\text{-strongly as } n \to \omega \right\} \\
\mathcal M^\omega(M) &= \left \{ (x_n)_n \in \ell^\infty(\N, M) :  (x_n)_n \, \mathcal I_\omega(M) \subset \mathcal I_\omega(M) \text{ and } \mathcal I_\omega(M) \, (x_n)_n \subset \mathcal I_\omega(M)\right\}.
\end{align*}

We have that the {\em multiplier algebra} $\mathcal M^\omega(M)$ is a C$^*$-algebra and $\mathcal I_\omega(M) \subset \mathcal M^\omega(M)$ is a norm closed two-sided ideal. Following \cite{Oc85}, we define the {\em ultraproduct von Neumann algebra} $M^\omega$ by $M^\omega = \mathcal M^\omega(M) / \mathcal I_\omega(M)$. We denote the image of $(x_n)_n \in \mathcal M^\omega(M)$ by $(x_n)^\omega \in M^\omega$. 

For all $x \in M$, the constant sequence $(x)_n$ lies in the multiplier algebra $\mathcal M^\omega(M)$. We will then identify $M$ with $(M + \mathcal I_\omega(M))/ \mathcal I_\omega(M)$ and regard $M \subset M^\omega$ as a von Neumann subalgebra. The map $\rE_\omega : M^\omega \to M : (x_n)^\omega \mapsto \sigma \text{-weak} \lim_{n \to \omega} x_n$ is a faithful normal conditional expectation. For every faithful normal state $\varphi \in M_\ast$, the formula $\varphi^\omega = \varphi \circ \rE_\omega$ defines a faithful normal state on $M^\omega$. Observe that $\varphi^\omega((x_n)^\omega) = \lim_{n \to \omega} \varphi(x_n)$ for all $(x_n)^\omega \in M^\omega$.

Let $Q \subset M$ be any von Neumann subalgebra with faithful normal conditional expectation $\rE_Q : M \to Q$. Choose a faithful normal state $\varphi$ on $Q$ and still denote by $\varphi$ the faithful normal state $\varphi \circ \rE_Q$ on $M$. We have $\ell^\infty(\N, Q) \subset \ell^\infty(\N, M)$, $\mathcal I_\omega(Q) \subset \mathcal I_\omega(M)$ and $\mathcal M^\omega(Q) \subset \mathcal M^\omega(M)$. We will then identify $Q^\omega = \mathcal M^\omega(Q) / \mathcal I_\omega(Q)$ with $(\mathcal M^\omega(Q) + \mathcal I_\omega(M)) / \mathcal I_\omega(M)$ and regard $Q^\omega \subset M^\omega$ as a von Neumann subalgebra. Observe that the norm $\|\cdot\|_{(\varphi |_Q)^\omega}$ on $Q^\omega$ is the restriction of the norm $\|\cdot\|_{\varphi^\omega}$ to $Q^\omega$. Observe moreover that $(\rE_Q(x_n))_n \in \mathcal I_\omega(Q)$ for all $(x_n)_n \in \mathcal I_\omega(M)$ and $(\rE_Q(x_n))_n \in \mathcal M^\omega(Q)$ for all $(x_n)_n \in \mathcal M^\omega(M)$. Therefore, the mapping $\rE_{Q^\omega} : M^\omega \to Q^\omega : (x_n)^\omega \mapsto (\rE_Q(x_n))^\omega$ is a well-defined conditional expectation satisfying $\varphi^\omega \circ \rE_{Q^\omega} = \varphi^\omega$. Hence, $\rE_{Q^\omega} : M^\omega \to Q^\omega$ is a faithful normal conditional expectation.

\begin{prop}\label{analytic-sequence}
Let $(M, \varphi)$ be any $\sigma$-finite von Neumann algebra endowed with a faithful normal state.  Let $a > 0$. Then for every $x \in M^\omega(\sigma^{\varphi^\omega}, [-a, a])$, there exists $(x_n)_n \in \mathcal M^\omega(M)$ such that $x = (x_n)^\omega$ and $x_n\in M(\sigma^{\varphi},[-4a,4a])$ for every $n \in \N$.
\end{prop}
\begin{proof}
Following \cite[Definition 4.12]{AH12}, denote by $D_a\colon \R\rightarrow \R$ the de la Vall$\rm\acute{e}e$ Poussin kernel defined by
\begin{equation*}
D_a(t) = \left\{ 
{\begin{array}{ll} \frac{\cos(at) - \cos(2at)}{\pi a t^2} & \text{ if } t \neq 0  \\ 
\frac{3 a}{ 2 \pi} & \text{ if } t = 0.
\end{array}} \right.
\end{equation*}
Its Fourier transform is given by
\begin{equation*}
\widehat {D_a}(\xi) = \left\{ 
{\begin{array}{ll} 1 & \text{ if } |\xi| \leq a  \\ 
2 - \frac{|\xi|}{2} & \text{ if } a \leq |\xi| \leq 2a \\
0 & \text{ if } |\xi| > 2a.
\end{array}} \right.
\end{equation*}

Let $(x_n)_n \in \mathcal M^\omega(M)$ be any element such that $x = (x_n)^\omega$. Since $x\in M^\omega(\sigma^{\varphi^\omega}, [-a, a])$ and since $\widehat{D_{2a}} =1$ on $[-2a,2a]$, we have $x=\sigma_{D_{2a}}^{\varphi^\omega}(x)$ by \cite[Lemma XI.1.3(iii)]{Ta03}. By \cite[Lemma 4.1.4]{AH12}, we have $(\sigma_{D_{2a}}^\varphi(x_n))_n \in \mathcal M^\omega(M)$ and $x=\sigma_{D_{2a}}^{\varphi^\omega}(x) = (\sigma_{D_{2a}}^\varphi(x_n))^\omega$. Since the support of $\widehat{D_{2a}}$ is contained in $[-4a,4a]$, we have $\sigma_{D_{2a}}^\varphi(x_n)\in M(\sigma^{\varphi},[-4a,4a])$ for every $n \in \N$. 
\end{proof}

\subsection*{Kirchberg's quotient weak expectation property (QWEP)}

Let $A \subset B$ be an inclusion of C$^*$-algebras. Following \cite{Ki92, Oz03}, we say that $A$ is {\em cp complemented} (resp.\ {\em weakly cp complemented}) in $B$ if there exists a ucp map $\Phi : B \to A$ (resp.\ $\Phi : B \to A^{\ast \ast}$) such that $\Phi |_{A} = \id_A$.

\begin{df}[{\cite{Ki92}}]
We say that a C$^*$-algebra $A$ has the {\em weak expectation property} (WEP) if it is weakly cp complemented in $\mathbf B(H)$ for some (or any) faithful representation $A \subset \mathbf B(H)$. We say that a C$^*$-algebra $A$ has the {\em quotient weak expectation property} (QWEP) if $A$ is the quotient of a C$^*$-algebra with the WEP. 
\end{df}

The QWEP conjecture states that all C$^*$-algebras have the QWEP. We refer the reader to \cite{Ki92, Oz03} for more on the QWEP conjecture.

\subsection*{Popa's local quantization principle}
In this Subsection, we review Popa's {\em local quantization principle} \cite{Po92, Po95a} that will play a crucial role in the proof of Lemma \ref{technical-lemma1}.

\begin{lem}\label{local-quantization-principle-lemma}
Let $(M, \varphi)$ be any $\sigma$-finite von Neumann algebra endowed with a faithful normal state and $Q \subset M^\varphi$ any von Neumann subalgebra. Let $X\subset M$ be any finite subset such that $\rE_{Q\vee (Q'\cap M)}(X)=0$ and $\{q_j\}_{j \in J}$ any finite partition of unity with projections $q_j \in Q$. 

Then for every $\delta>0$, there exists a finite partition of unity $\{p_i\}_{i \in I}$ with projections $p_i \in Q$ that refines the partition $\{q_j\}_{j \in J}$ and such that $\|\sum_{i \in I} p_ixp_i\|_\varphi<\delta$ for all $x\in X$. 
\end{lem}
\begin{proof}
See the proofs of \cite[Lemma A.1.1]{Po92} and \cite[Lemma A.1.1]{Po95a}.
\end{proof}

\begin{prop}\label{local-quantization-principle}
Let $(M, \varphi)$ be any $\sigma$-finite von Neumann algebra endowed with a faithful normal state and $Q \subset M^\varphi$ any type ${\rm II_1}$ subfactor. Let $\varepsilon > 0$ and $X \subset M$ be any finite subset. Then there exists a nonzero projection $q \in Q$ such that 
$$\|q x q - \rE_{Q'\cap M}(x) q\|_\varphi < \varepsilon \|q\|_\varphi, \forall x \in X.$$
\end{prop}

\begin{proof}
Since the proof is the same as the one of \cite[Theorem A.1.2]{Po92}, we will only sketch it. For every $x \in X$, write $x' = \rE_{Q\vee(Q'\cap M)}(x)$ and $x\dpr =x - x'$. Put $Y = \{x\dpr: x \in X\}$. Note that given any $\varepsilon_1 > 0$, every $x'$ can be approximated using finitely many $b\in Q$ and $b'\in Q'\cap M$ such that $\|x'-\sum bb' \|_\varphi<\varepsilon_1$. We denote by $\mathcal{S}\subset Q$ the set of all such $b \in Q$.  Since all the arguments are done inside the type ${\rm II_1}$ subfactor $Q$, proceeding exactly as in the proof of \cite[Theorem A.1.2, pp.\ 246-247]{Po92}, for every $\varepsilon_2>0$, one can find a finite dimensional abelian unital $\ast$-subalgebra $A_2\subset Q$ such that $\|\rE_{A_2'\cap Q}(b)-\varphi(b)1 \|_\varphi<\varepsilon_2$ for all $b\in \mathcal{S}$.

Since $A_2$ is obtained from a partition of unity with projections in $Q$ and since we have $\rE_{Q\vee(Q'\cap M)}(x\dpr) = 0$, Lemma \ref{local-quantization-principle-lemma} implies that for every $\varepsilon_3 > 0$, we can find a refinement $\{q_i\}_{i \in I}\subset Q$ of the partition coming from $A_2$ such that with $A_3=\sum_{i \in I}\C q_i$, we have $\|\rE_{A_3'\cap M}(x\dpr)\|_\varphi<\varepsilon_3$ for all $x\dpr \in Y$. Observe that we have $\rE_{A_3' \cap M}(x) = \sum_{i \in I} q_i x q_i$ and $Q' \cap M \subset A_3' \cap M \subset A_2' \cap M$. Then we have 
\begin{align*}
\|\rE_{A_3'\cap M}(x)-\rE_{Q'\cap M}(x)\|_\varphi
&\leq\|\rE_{A_3'\cap M}(x')-\rE_{Q'\cap M}(x')\|_\varphi+\|\rE_{A_3'\cap M}(x\dpr)-\rE_{Q'\cap M}(x\dpr)\|_\varphi\\
&=\|\rE_{A_3'\cap M}(x'-\rE_{Q'\cap M}(x'))\|_\varphi+\|\rE_{A_3'\cap M}(x\dpr)\|_\varphi\\
&\leq\|\rE_{A_2'\cap M}(x')-\rE_{Q'\cap M}(x')\|_\varphi+\|\rE_{A_3'\cap M}(x\dpr)\|_\varphi\\
&< \left\|\sum \rE_{A_2'\cap M}(bb')-\rE_{Q'\cap M}(bb') \right\|_\varphi+2\varepsilon_1+\varepsilon_3\\
&\leq\sum \|b'\|_\infty \|\rE_{A_2'\cap M}(b)-\rE_{Q'\cap M}(b)\|_\varphi+2\varepsilon_1+\varepsilon_3\\
&= \sum \|b'\|_\infty \|\rE_{A_2'\cap Q}(b)-\varphi(b)1\|_\varphi+2\varepsilon_1+\varepsilon_3\\
&<\sum \|b'\|_\infty \, \varepsilon_2+2\varepsilon_1+\varepsilon_3.
\end{align*}

So, by choosing $\varepsilon_1, \varepsilon_2, \varepsilon_3$ sufficiently small so that $\sum \|b'\|_\infty \, \varepsilon_2+2\varepsilon_1+\varepsilon_3 < \frac{\varepsilon}{\sqrt{|X|}}$, we get
\begin{align*}
\sum_{i \in I} \|q_i(x-\rE_{Q'\cap M}(x))q_i\|_\varphi^2
&= \left\|\sum_{i \in I} q_i(x-\rE_{Q'\cap M}(x))q_i \right\|_\varphi^2 \\
&= \|\rE_{A_3'\cap M}(x)-\rE_{Q'\cap M}(x)\|_\varphi^2 \\
&< \sum_{i \in I} \frac{\varepsilon^2}{|X|}\|q_i\|_\varphi^2
\end{align*}
and hence
$$\sum_{i \in I}\left(\sum_{x\in X}\|q_i(x-\rE_{Q'\cap M}(x))q_i\|_\varphi^2 \right)< \sum_{i \in I} \varepsilon^2\|q_i\|_\varphi^2.$$
Thus, there exists $i \in I$ such that $\sum_{x\in X}\|q_i(x-\rE_{Q'\cap M}(x))q_i\|_\varphi^2< \varepsilon^2\|q_i\|_\varphi^2$.
\end{proof}

\section{Popa's incremental patching method in arbitrary von Neumann algebras}\label{incremental}

\begin{nota}
Let $B$ be any von Neumann algebra, $v \in B$ any partial isometry satisfying $v^*v = vv^*$ and $X \subset B$ any subset. For $k \geq 1$, put 
$$X_v^k = \left\{x_0 \prod_{i = 1}^k v_i x_i : v_i = v \text{ or } v_i = v^*, x_0, x_k \in X \cup \{1\} \text{ and } x_1, \dots , x_{k - 1} \in X\right\}.$$
\end{nota}

For all $a > 0$, put $C(a) = C(a, {\rm i}/2) = C(a, - {\rm i}/2) > 1$ where $(0, +\infty) \times \C \to (0, +\infty) : (a, z) \mapsto C(a, z) := 2 \exp(2a |\Im(z)|) + \exp(a |\Im(z)|)$ is the function that appeared in Lemma \ref{analytic}. In the rest of this Section, we assume that $(M, \varphi)$ is a non-type ${\rm I}$ $\sigma$-finite factor endowed with a faithful normal state and $Q \subset M^\varphi$ is any von Neumann subalgebra satisfying $Q' \cap M = \C 1$. Observe that $Q$ is necessarily a type ${\rm II_1}$ subfactor. We also fix a non-principal ultrafilter $\omega \in \beta(\N) \setminus \N$.

In Lemma \ref{technical-lemma1}, we start by proving the existence of nonzero partial isometries in $Q$ with good approximate independence properties in $M$ with respect to the state $\varphi$ for words of bounded length and with letters in bounded spectral subspaces of $(\sigma_t^\varphi)$. This is an analogue of \cite[Lemma 1.2]{Po95b}.

\begin{lem}\label{technical-lemma1}
Let $\varepsilon > 0$, $a > 0$, $n \geq 1$, $X \subset M(\sigma^\varphi, [-a, a])$ and $Y\subset M$ be any finite subsets such that $X = X^*$ and $Y = Y^*$. Let $f\in Q$ be any nonzero projection and assume that $\varphi(fXf)=0$. 

Then there exists a partial isometry $v \in fQf$ satisfying $v^*v = vv^*$, $\varphi(v^*v) > \frac{\varphi(f)}{12 \, C(a)^2+1}$ and
$$\max\{|\varphi(fxf) |,|\varphi(x) |\}\leq \varepsilon, \forall x \in \bigcup_{k = 1}^n X_v^k \quad \text{and} \quad |\varphi(yvz)|\leq\varepsilon, \forall y,z\in Y.$$
\end{lem}

\begin{proof}
The proof follows very closely the one of \cite[Lemma 1.2]{Po95b}. We will nevertheless give a detailed proof for the reader's convenience. We may assume that $X\subset \Ball(M)$. We fix $\varepsilon_0 > 0$ and define inductively $\varepsilon_k$ for all $1 \leq k \leq n$ by the formula $\varepsilon_k= 2^k C(a)\varepsilon_{k-1}$. We moreover choose $\varepsilon_0 > 0$ sufficiently small so that $\varepsilon_n<\varepsilon$. 

We next define $\mathcal{W}$ as the set of all partial isometries $v$ in $fQf$ such that 
\begin{itemize}
	\item $v^*v=vv^*$,
	\item $\max\{|\varphi(fxf)|,|\varphi(x)|\}\leq\varepsilon_k\varphi(v^*v)$ for all $1 \leq k \leq n$ and all $x\in X_v^k$ and
	\item $|\varphi(yvz)|\leq\varepsilon\varphi(v^*v)$ for all $y,z\in Y$.
\end{itemize}
Note that $\mathcal{W}$ is not empty since $0 \in \mathcal W$. We define a partial order on $\mathcal{W}$ by $w_1\leq w_2$ $\Leftrightarrow$ $w_1=w_2w_1^*w_1$. Then $(\mathcal{W}, \leq)$ is an inductive set. By Zorn's lemma, take a maximal element $v\in\mathcal{W}$ and put $p=f-v^*v$. We will show that $\varphi(v^*v)> \varphi(f)/(12 \, C(a)^2+1)$.

Suppose by contradiction that $\varphi(v^*v)\leq \varphi(f)/(12 \, C(a)^2+1)$ and note that this implies that $\varphi(v^*v)/\varphi(p)\leq 1/(12 \, C(a)^2)$. Our goal is to find a non-zero partial isometry $w\in pQp$ such that $v+w\in \mathcal{W}$. This in turn will contradict the maximality of $v \in \mathcal W$. To do so, we first fix a nonzero zero projection $q \in pQp$ and take any unitary $w\in qQq$. Put $u=v+w$.

Let $1 \leq k \leq n$. Take $x=x_0 \prod_{i = 1}^k u_i x_i\in X_u^k$ and decompose it using $u_i=v_i+w_i$ as follows:
$$x=x_0 \prod_{i = 1}^k v_i x_i + \sum_{\ell=1}^k\sum_{i\in I_\ell}z_0^i \prod_{j = 1}^\ell w_j^i z_j^i,$$
where $I_\ell=\{(i_1,\ldots,i_\ell) : 1\leq i_1<\cdots <i_\ell\leq k \}$, $w_j^i=w_{i_j}$, $z_0^i=x_0v_1x_1\cdots x_{i_1-1}$, $z_j^i=x_{i_j}v_{i_j+1}x_{i_j+1}\cdots v_{i_{j+1}}x_{i_{j+1}-1}$ for $1\leq j<\ell$ and $z_\ell^i=x_{i_\ell}v_{i_\ell+1}\cdots v_kx_k$. By applying $\varphi$ and using the triangle inequality, we obtain 
\begin{align*}
|\varphi(x)|&\leq
\left|\varphi\left(x_0 \prod_{i = 1}^k v_i x_i \right) \right| + \sum_{\ell=1}^k\sum_{i\in I_\ell} \left|\varphi\left(z_0^i \prod_{j = 1}^\ell w_j^i z_j^i\right)\right| \\
|\varphi(fxf)|&\leq
\left|\varphi\left(f\, x_0 \prod_{i = 1}^k v_i x_i \, f\right)\right| + \sum_{\ell=1}^k\sum_{i\in I_\ell} \left|\varphi\left(f \, z_0^i \prod_{j = 1}^\ell w_j^i z_j^i \,f\right)\right|.
\end{align*}
The first terms on the right hand side are less than or equal to $\varepsilon_k\varphi(v^*v)$ since $v \in \mathcal W$. So our task is to find an appropriate projection $q \in pQp$ and then an appropriate $w \in \mathcal U(q Q q)$ such that the second terms on the right hand side are less than or equal to $\varepsilon_k\varphi(w^*w)$. From that, we will obtain that $\max\{|\varphi(fxf)|,|\varphi(x)|\}\leq  \varepsilon_k\varphi(v^*v) +  \varepsilon_k\varphi(w^*w) = \varepsilon_k\varphi(u^*u)$.

\begin{claim}
There is a nonzero projection $q\in pQp$ such that
$$\sum_{\ell=2}^k\sum_{i\in I_\ell}\max\left\{\left|\varphi\left(f\, z_0^i \prod_{j = 1}^\ell w_j^i z_j^i \, f\right)\right|, \left|\varphi \left(z_0^i \prod_{j = 1}^\ell w_j^i z_j^i \right) \right| \right\}<\frac{\varepsilon_k}{2}\varphi(q)$$
for all $2 \leq k\leq n$, all $w_j^i \in \{w, w^*\}$ for any unitary $w\in \mathcal U(qQq)$ and all $z_j^i$ that are constructed from any $x_j\in X$ and $v_j$ as above.
\end{claim}

\begin{proof}[Proof of the Claim]
This is the main part of the proof of Lemma \ref{technical-lemma1}. We will make use of Popa's local quantization principle (see Proposition \ref{local-quantization-principle}). Recall that $\|yb\|_\varphi\leq \|\sigma_{-{\rm i}/2}^\varphi(b^*)\|_\infty\|y\|_\varphi$ for all $y\in M$ and all analytic $b \in M$ (see \cite[Lemma VIII 3.10 (i)]{Ta03}). Moreover, Proposition \ref{analytic} implies that $\|\sigma_{-{\rm i}/2}^\varphi(b^*)\|_\infty \leq C(a) \|b\|_\infty$ for all $b\in M(\sigma^\varphi, [-a, a])$. 

For all $\ell\geq2$ and all $i = (i_1, \dots, i_\ell)\in I_\ell$, we have
\begin{align*}
\|(fz_0^iq)^*\|_\varphi\leq\|(z_0^iq)^*\|_\varphi
&= \|q x_{i_1-1}^*\cdots x_1^*v_1^*x_0^*\|_\varphi\\
&\leq \|\sigma_{-{\rm i}/2}^\varphi(x_0)\|_\infty \|v_1\|_\infty \cdots \|\sigma_{-{\rm i}/2}^\varphi(x_{i_1 - 1})\|_\infty \|q\|_\varphi\\
&\leq C(a)^{i_1} \|q\|_\varphi 
\end{align*}
and
\begin{align*}
\left\|\prod_{j = 1}^\ell w_j^i z_j^i \, f\right\|_\varphi\leq \left\|\prod_{j = 1}^\ell w_j^i z_j^i \right\|_\varphi
&=\|w^i_1\, qz^i_1q \, w_2^iz_2^i\cdots w_\ell^iz_\ell^i\|_\varphi\\
&\leq \|\sigma_{-{\rm i}/2}^\varphi(z_\ell^{i*})\|_\infty \|w_\ell^{i*}\|_\infty \cdots \|\sigma_{-{\rm i}/2}^\varphi(z_2^{i*})\|_\infty \|w_2^{i*}\|_\infty \|w^i_1\|_\infty  \|qz^i_1q\|_\varphi\\
&\leq C(a)^{k-i_2+1} \|qz^i_1q\|_\varphi.
\end{align*}
Then by the Cauchy-Schwarz inequality,  we obtain 
\begin{align*}
\max\left\{\left|\varphi \left(f \, z_0^i\prod_{j = 1}^\ell w_j^i z_j^i \, f\right)\right|, \left|\varphi \left(z_0^i\prod_{j = 1}^\ell w_j^i z_j^i \right)\right| \right\}&\leq \|(z_0^iq)^*\|_\varphi \left\|\prod_{j = 1}^\ell w_j^i z_j^i\right\|_\varphi \\
&\leq C(a)^{k+i_1-i_2+1} \|q\|_\varphi \|qz^i_1q\|_\varphi.
\end{align*}
We now apply Proposition \ref{local-quantization-principle} to the inclusion $pQp\subset pMp$, for some $\alpha>0$ and for all possible $pz^i_1p$ and we obtain a nonzero projection $q\in pQp$ satisfying
$$\|qz_1^iq - \varphi_p(pz_1^ip)q\|_\varphi<\alpha \|q\|_\varphi \quad \text{for all possible } pz^i_1p$$ 
where $\varphi_p= \frac{\varphi(p \cdot p)}{\varphi(p)}$. By choosing $\alpha > 0$ sufficiently small, we have 
$$\max\left\{\left|\varphi \left(f \, z_0^i\prod_{j = 1}^\ell w_j^i z_j^i \, f\right)\right|, \left|\varphi \left(z_0^i\prod_{j = 1}^\ell w_j^i z_j^i \right)\right| \right\}< \frac{1}{2^k}\frac{\varepsilon_k}{6}\|q\|_\varphi^2+ C(a)^{k+i_1-i_2+1} \|q\|_\varphi^2 |\varphi_p(pz_1^ip)|.$$

Next, put $m=i_2-i_1-1$. Using the facts that $p,v\in M^\varphi$, $z_1^i\in X$ when $m = 0$ and $\varphi(f X f) = 0$, $z_1^i\in X_v^m$ when $m\geq1$ and $vz_1^iv^*\in X_v^{m+2}$, we have 
\begin{align*}
|\varphi(pz_1^ip)|
&=|\varphi(pz_1^i)| \\
&=|\varphi((f-v^*v)z_1^i)|\\
&\leq|\varphi(fz_1^if)|+|\varphi(v^*vz_1^i)| \\
&= |\varphi(fz_1^if)|+|\varphi(vz_1^iv^*)|\\
&\leq\varepsilon_m\varphi(v^*v)+\varepsilon_{m+2}\varphi(v^*v)
< 2\varepsilon_{m+2}\varphi(v^*v)\\
&\leq \frac{\varepsilon_{m+2}}{6}\frac{1}{C(a)^2}\varphi(p)
\end{align*}
and hence $|\varphi_p(pz_1^ip)| \leq  \frac{\varepsilon_{m+2}}{6}\frac{1}{C(a)^2}$. Since $\varepsilon_{m+2}\leq C(a)^{-k+m+2} 2^{-k}\varepsilon_k$ when $m+2<k$ and since $m+2=k$ happens only when $\ell=2$ and $(i_1, i_2)=(1, k)$, we have 
\begin{align*}
& \sum_{\ell=2}^k\sum_{i\in I_\ell}\max\left\{\left|\varphi \left(f \, z_0^i\prod_{j = 1}^\ell w_j^i z_j^i \, f\right)\right|, \left|\varphi \left(z_0^i\prod_{j = 1}^\ell w_j^i z_j^i \right)\right| \right\}\\
&< \sum_{\ell=2}^k\sum_{i\in I_\ell}\frac{1}{2^k}\frac{\varepsilon_k}{6}\|q\|_\varphi^2+ \sum_{\ell=2}^k\sum_{i\in I_\ell}\frac{\varepsilon_{m+2}}{6}C(a)^{k+i_1-i_2-1} \|q\|_\varphi^2 \\
&<\frac{\varepsilon_k}6\|q\|_\varphi^2+ \left(\sum_{\ell=2}^k\sum_{i\in I_\ell, (i_1,i_2)\neq(1,k)}\frac{1}{2^k}\frac{\varepsilon_k}{6}\|q\|_\varphi^2 \right) + \frac{\varepsilon_k}6\|q\|_\varphi^2\\
&<\frac{\varepsilon_k}{2}\|q\|_\varphi^2.
\end{align*}
Note that we used the fact that $\sum_{\ell=2}^k\sum_{i\in I_\ell} 1=\sum_{\ell=2}^k\begin{pmatrix}
k\\
\ell
\end{pmatrix}
=2^k-k-1<2^k$. Thus, we have obtained the desired projection $q\in pQp$. This finishes the proof of the Claim.
\end{proof}

We fix now a projection $q\in pQp$ as in the Claim and we next find a unitary $w\in \mathcal U(qQq)$ that satisfies
$$\max \left\{\sum_{i\in I_1}|\varphi(f \, z_0^i w_1^i z_1^i \, f)|,\sum_{i\in I_1}|\varphi(z_0^i w_1^i z_1^i)| \right\}\leq\frac{\varepsilon_k}{2}\|q\|_\varphi^2 \quad \text{and} \quad |\varphi(ywz)|\leq\varepsilon\|q\|_\varphi^2$$
for all possible $z_0^i, z_1^i$ and all $y,z\in Y$. Since $qQq$ is diffuse and since $Y \cup \{z_0^i, z_1^i : i \in I_1\}$ is finite, we can find a unitary $w\in \mathcal U( qQq)$ such that $|\varphi(fz_0^i w z_1^if)|$, $|\varphi(z_0^i w z_1^i)|$ and $|\varphi(ywz)|$ are small enough so that we have the desired inequality. 

As we mentioned before the Claim, we obtain that $\max\{|\varphi(fxf)|,|\varphi(x)|\}\leq\varepsilon_k\varphi(u^*u)$ for all $1 \leq k\leq n$ and all $x\in X_u^k$ and $|\varphi(yuz)|\leq\varepsilon\varphi(u^*u)$ for all $y, z \in Y$. Thus $u\in \mathcal{W}$. Since $v \leq u$ and $v \neq u$, this  contradicts the maximality of $v \in \mathcal W$ and finishes the proof of Lemma \ref{technical-lemma1}.
\end{proof}

From now on, we will be working in the ultraproduct framework $M^\omega$. In Lemma \ref{technical-lemma2}, we construct nonzero partial isometries in $Q^\omega$ with good independence properties in $M^\omega$ with respect to the ultraproduct state $\varphi^\omega$ for words with letters in bounded spectral subspaces of $(\sigma_t^{\varphi^\omega})$. This is an analogue of \cite[Lemma 1.3]{Po95b}.

\begin{lem}\label{technical-lemma2}
Let $a > 0$, $X \subset M^\omega(\sigma^{\varphi^\omega}, [-a, a])$ and $Y\subset M^\omega$ be any countable subsets such that $X = X^*$ and $Y = Y^*$. Let $f\in Q^\omega$ be any nonzero projection  such that $\varphi^\omega(fXf)=0$. Then there exists a partial isometry $v \in fQ^\omega f$ satisfying $v^*v = vv^*$, $\varphi^\omega(v^*v) \geq \frac{\varphi^\omega(f)}{12 \, C(4a)^2+1}$ and  
$$\varphi^\omega(x) = 0, \forall k \geq 1, \forall x \in X_v^k \quad \text{and} \quad \varphi^\omega(yvz)=0, \forall y,z\in Y.$$
\end{lem}

\begin{proof}
Write $X=\{x^\ell : \ell\in \N\}$, $Y=\{y^\ell : \ell\in \N\}$, $x^\ell=(x_n^\ell)^\omega$ and $y^\ell=(y_n^\ell)^\omega$ for some $(x^\ell_n)_n,(y_n^\ell)_n\in \mathcal M^\omega(M)$. We also write $f=(f_n)^\omega$ for some nonzero projections $f_n\in Q$. By Proposition \ref{analytic-sequence}, we may assume that $x^\ell_n\in M(\sigma^{\varphi}, [-4a, 4a])$ for every $n \in \N$. We also have $f_nx^\ell_nf_n\in M(\sigma^{\varphi}, [-4a, 4a])$ for every $n \in \N$. Also, since $\varphi^\omega(fXf) = 0$, up to replacing $x^\ell_n$ by $x^\ell_n-\frac{\varphi(f_nx^\ell_nf_n)}{\varphi(f_n)} 1 \in M(\sigma^\varphi, [-4a, 4a])$, we may assume that $\varphi(f_nx^\ell_nf_n)=0$ for all $\ell, n \in \N$. 

We can then apply Lemma \ref{technical-lemma1} to $\varepsilon_n = \frac{1}{n}>0$, $f_n \in Q$, $X_n=\{x_n^\ell : \ell\leq n\}$ and $Y_n=\{y_n^\ell : \ell\leq n\}$ to find a partial isometry $v_n\in f_nQf_n$ such that $v_n^*v_n=v_nv_n^*$, $\varphi(v_n^*v_n)>\frac{\varphi(f_n)}{12 \, C(4a)^2+1}$, $|\varphi(x) | \leq \frac{1}{n}$ for all $x \in \bigcup_{k = 1}^n (X_n)_{v_n}^k$ and $|\varphi(yv_nz)|\leq\frac{1}{n}$ for all $y,z\in Y_n$. 

Then $v=(v_n)^\omega\in Q^\omega$ is a partial isometry satisfying $v^*v=vv^*$, $\varphi^\omega(v^*v)=\lim_{n\rightarrow\omega}\varphi(v_n^*v_n)\geq\lim_{n\rightarrow\omega}\frac{\varphi(f_n)}{12 \, C(4a)^2+1}=\frac{\varphi^\omega(f)}{12 \, C(4a)^2+1}$ and $\varphi^\omega(yvz)=0$ for all $y,z\in Y$. As usual, put $(v_n)_i=v_n$ or $v_n^*$. Moreover, for all $k\geq 1$ and all $x=x_0 \prod_{i = 1}^k v_i x_i\in X_v^k$ with $x_i=x^{\ell_i}$ for some $\ell_i \in \N$, using the fact that $|\varphi(x^{\ell_0}_n \prod_{i = 1}^k (v_n)_i  x_n^{\ell_i})|\leq \frac{1}{n}$ whenever $n\geq \max\{k, \ell_i : i \leq k\}$, we have
$$|\varphi^\omega(x) |=\lim_{n\rightarrow\omega}\left|\varphi\left(x^{\ell_0}_n \prod_{i = 1}^k (v_n)_i  x_n^{\ell_i}\right)\right| \leq\lim_{n\rightarrow\omega}\frac{1}{n} =0.$$ 
This finishes the proof of Lemma \ref{technical-lemma2}. 
\end{proof}

Using a maximality argument, we construct in Lemma \ref{technical-lemma3} unitaries in $\mathcal U(Q^\omega)$ with good independence properties in $M^\omega$ with respect to the ultraproduct state $\varphi^\omega$ for words of bounded length and with letters in bounded spectral subspaces of $(\sigma_t^{\varphi^\omega})$. This is an analogue of \cite[Lemma 1.4]{Po95b}.

\begin{lem}\label{technical-lemma3}
Let $a > 0$, $n \geq 2$, $X \subset (M^\omega \ominus \C1) \cap M^\omega(\sigma^{\varphi^\omega}, [-a, a])$ and $Y\subset M^\omega$ be any countable subsets such that $X = X^*$ and $Y = Y^*$. Then there exists a unitary $v \in \mathcal U(Q^\omega)$ satisfying
$$\varphi^\omega(x) = 0, \forall x \in \bigcup_{k = 1}^n X_v^k \quad \text{and} \quad \varphi^\omega(yvz)=0, \forall y,z\in Y.$$
\end{lem}

\begin{proof}
Define $\mathcal{W}$ as the set of all partial isometries $v$ in $Q^\omega$ such that 
\begin{itemize}
	\item $v^*v=vv^*$,
	\item $|\varphi^\omega(x)|=0$ for all $1 \leq k \leq n$ and all $x\in X_v^k$ and
	\item $|\varphi^\omega(yvz)|=0$ for all $y,z\in Y$.
\end{itemize}
Note that $\mathcal W$ is not empty since $0 \in \mathcal W$. We define a partial order on $\mathcal W$ by $w_1\leq w_2$ $\Leftrightarrow$ $w_1=w_2w_1^*w_1$. Then $(\mathcal W, \leq)$ is an inductive set. By Zorn's lemma, take a maximal element $v\in \mathcal{W}$. We show that $v$ is a unitary element that satisfies the conclusion of Lemma \ref{technical-lemma3}.

Suppose by contradiction that $v$ is not a unitary and put $p=1-v^*v\neq0$. Since $v\in\mathcal{W}$, $n \geq 2$ and $\varphi^\omega(X) = 0$, we have $\varphi^\omega(p X p) = 0$ and $\varphi^\omega(pX_v^kp)= 0$ for all $1\leq k\leq n-2$. Moreover, we have $X_v^k\subset M^\omega(\sigma^{\varphi^\omega},[-(n - 1)a,(n - 1)a])$ for all $k\leq n - 2$. So, we can apply Lemma \ref{technical-lemma2} to $(n - 1)a$, $p \in Q^\omega$, $X\cup\bigcup_{k = 1}^{n-2} X_v^k$ and $Y\cup \bigcup_{k=1}^{n-1}X_v^k$ to obtain a nonzero partial isometry $w\in pQ^\omega p$ that satisfies the conclusion of Lemma \ref{technical-lemma2}. 

Put $u=v+w$. Then we have $u^*u=uu^*$ and $\varphi^\omega(yuz)=0$ for all $y,z\in Y$. Moreover, for all $1 \leq k \leq n$ and all $x=x_0\prod_{i=1}^ku_ix_i \in X_u^k$, we have
$$\varphi^\omega(x)=\varphi^\omega\left(x_0 \prod_{i = 1}^k v_i x_i \right) + \sum_{\ell=1}^k\sum_{i\in I_\ell}\varphi^\omega\left(z_0^i \prod_{j = 1}^\ell w_j^i z_j^i\right)$$
where each $z^i_j$ is given as in the proof of Lemma \ref{technical-lemma1}.

Since $v\in\mathcal{W}$, we have $\varphi^\omega(x_0 \prod_{i = 1}^k v_i x_i) = 0$. Next, when $\ell\geq 2$, each $z_j^i$ belongs to $X$ or $X_v^k$ for some $1\leq k\leq n-2$ and hence $\varphi^\omega(z_0^i \prod_{j = 1}^\ell w_j^i z_j^i)=0$ by the choice of the partial isometry $w$ and the first equality guaranteed by Lemma \ref{technical-lemma2}. Finally, when $\ell=1$, each $z_0^i$ and each $z_1^i$  belongs to $X$ or $X_v^{k}$ for some $1\leq k\leq n-1$ and hence $\varphi^\omega(z_0^i w_1^i z_1^i)=0$ by the choice of the partial isometry $w$ and the second equality guaranteed by Lemma \ref{technical-lemma2}. Thus $\varphi^\omega(x)=0$ and so $u \in \mathcal{W}$. Since $v \leq u$ and $v \neq u$, this contradicts the maximality of $v \in \mathcal W$ and finishes the proof of Lemma \ref{technical-lemma3}.
\end{proof}

Using a diagonal procedure, we finally construct in Lemma \ref{technical-lemma4} unitaries in $\mathcal U(Q^\omega)$ with good independence properties in $M^\omega$ with respect to the ultraproduct state $\varphi^\omega$. This is a new step compared to the strategy developed in \cite{Po95b}.

\begin{lem}\label{technical-lemma4}
Let $X \subset (M^\omega \ominus \C1) \cap (\cup_{n\in\N}M^\omega(\sigma^{\varphi^\omega}, [-n, n]))$ be any countable subset such that $X = X^*$. Then there exists a unitary $v \in \mathcal U(Q^\omega)$ such that 
$$\varphi^\omega(x)  = 0, \forall k \geq 1, \forall x \in X_v^k.$$
\end{lem}

\begin{proof}
The proof is similar to the one of Lemma \ref{technical-lemma2}. Write $X=\{x^\ell : \ell\in \N\}$ and $x^\ell=(x_n^\ell)^\omega\in M^\omega$ for some $(x^\ell_n)_n \in \mathcal M^\omega(M)$. Then for every $n \in \N$, the subset $\{x^\ell : \ell \leq n\}$ is contained in $M^\omega(\sigma^{\varphi^\omega}, [-a_n, a_n])$ for some $a_n>0$. So, proceeding exactly as in the proof of Lemma \ref{technical-lemma2}, we may assume that the subset $X_n=\{x_n^\ell : \ell \leq n\}$ is contained in $(M \ominus \C1)\cap M(\sigma^\varphi, [-4a_n, 4a_n])$. 

Since $M \subset M^\omega$ and $\sigma_t^{\varphi^\omega} |_M = \sigma_t^\varphi$ for all $t \in \R$, we may regard each subset $X_n$ as a subset of $M^\omega$ and we have $X_n \subset (M^\omega \ominus \C1) \cap M^\omega(\sigma^{\varphi^\omega}, [-4a_n, 4a_n])$ for every $n \in \N$. Then for every $n \in \N$, we choose a unitary $v^n\in \mathcal U(Q^\omega)$ that satisfies the conclusion of Lemma \ref{technical-lemma3} for the subset $X_n$. Write $v^n=(v^n_m)^\omega$ for some $v_m^n \in \Ball(Q)$. Observe that $\lim_{m \to \omega} \|1 - (v_m^n)^*v_m^n\|_\varphi = 0$ for every $n \in \N$. We will now construct a new unitary element in $\mathcal U(Q^\omega)$ that satisfies the conclusion of Lemma \ref{technical-lemma4} by choosing carefully ``diagonal'' elements from $v^n = (v_m^n)^\omega \in \mathcal U(Q^\omega)$. 

By the choice of $v^n$, for all $1 \leq k\leq n$ and all $x=x_0\prod_{i=1}^k(v^n)_ix_i\in(X_n)_{v^n}^k$, we have
$$0=\varphi^\omega(x) = \lim_{m\rightarrow\omega} \varphi\left(x_0\prod_{i=1}^k(v^n_m)_ix_i\right).$$
So, for every $n>0$, there exists $m_n \in \N$ large enough so that $u_n = v^n_{m_n}\in \Ball(Q)$ satisfies 
\begin{itemize}
\item $|\varphi(x_0\prod_{i=1}^k(u_n)_ix_i)|<\frac{1}{n}$ for all $1 \leq k\leq n$ and all $x_0,x_k\in \{1\}\cup X_n$, $x_i\in X_n$ and $(u_n)_i=u_n$ or $u_n^*$ and
\item $\|1-u_n^*u_n\|_\varphi < \frac{1}{n}$.
\end{itemize}
 Then we have $u=(u_n)^\omega\in  \mathcal U(Q^\omega)$ by the second item and the fact that $u_n \in Q \subset M^\varphi$. Moreover, by the first item, for all $k \geq 1$ and all $x=x_0 \prod_{i = 1}^k u_i x_i\in X_u^k$ with $x_i=x^{\ell_i}$ for some $\ell_i \in \N$, using the fact that $|\varphi(x^{\ell_0}_n \prod_{i = 1}^k (u_n)_i  x_n^{\ell_i})|\leq \frac{1}{n}$ whenever $n\geq \max \{k, \ell_i : i \leq k \}$, we have
$$|\varphi^\omega(x) |=\lim_{n\rightarrow\omega}\left|\varphi\left(x^{\ell_0}_n \prod_{i = 1}^k (u_n)_i  x_n^{\ell_i}\right) \right| \leq\lim_{n\rightarrow\omega}\frac{1}{n} =0.$$
This finishes the proof of Lemma \ref{technical-lemma4}.
\end{proof}

\begin{rem} 
We conclude this Section with a few remarks.
\begin{enumerate}
\item In Lemma \ref{technical-lemma4}, we can actually find $v \in \mathcal U(Q^\omega)$ such that $\varphi^\omega(yvz)=0$ for all $y,z\in Y$ where $Y\subset M^\omega$ is any given countable subset. Also, we can construct $v \in \mathcal U(Q^\omega)$ to be a Haar unitary by using the subsets $X_v^{k,n}$ as in \cite[Lemma 1.2]{Po95b}. However we will not use this observation in this article.

\item While Lemmas \ref{technical-lemma2} and \ref{technical-lemma3} could be proven in the more general framework of ultraproduct von Neumann algebras $(M_n, \varphi_n)^\omega$ with $(M_n, \varphi_n)$ a non-type ${\rm I}$ $\sigma$-finite factor endowed with a faithful normal state and $Q_n \subset M_n^{\varphi_n}$ any von Neumann subalgebra satisfying $Q_n' \cap M_n = \C 1$ for every $n \in \N$, the proof of Lemma \ref{technical-lemma4} does require the sequence $M_n = M$ to be constant as we need to regard $M \subset M^\omega$ as a von Neumann subalgebra.
\end{enumerate}
\end{rem}

\section{Proofs of the main results}\label{proofs}

\subsection*{Proof of Theorem \ref{thmA}}

Let $i \in \{1, 2\}$. Since $P_i$ has separable predual and is globally invariant under the modular automorphism group $(\sigma_t^{\varphi^\omega})$, using the proof of \cite[Proposition 4.11]{AH12}, we may find a countable
subset $X_{i} \subset \Ball(P_i \ominus \C 1) \cap (\cup_{n \in \N} M^\omega(\sigma^{\varphi^\omega}, [-n, n]))$ satisfying $X_{i}^* = X_{i}$ and such that $X_i$ is $\|\cdot\|_{\varphi^\omega}$-dense in $\Ball(P_i \ominus \C 1)$. Applying Lemma \ref{technical-lemma4} to the countable subset $X = X_1 \cup X_2$, there exists a unitary $v \in \mathcal U(Q^\omega)$ such that 
$$\varphi^\omega(x_0 \, vy_1v^* \cdots x_{k - 1} \,  v y_{k} v^* \, x_k) = 0$$
for all $k \geq 1$, all $x_0, x_k \in X_1 \cup \{1\}$, all $x_1, \dots, x_{k - 1} \in X_1$ and all $y_1, \dots, y_{k} \in X_2$. By $\|\cdot\|_{\varphi^\omega}$-density of $X_i$ in $\Ball(P_i \ominus \C 1)$ for all $i \in \{1, 2\}$, we obtain that $P_1$ and $v P_2 v^*$ are $\ast$-free inside $M^\omega$ with respect to the state $\varphi^\omega$. This finishes the proof of Theorem \ref{thmA}.

\subsection*{Proof of Corollary \ref{corB}}

Let $i \in \{1, 2\}$. There exists a normal $\ast$-embedding $\pi_i : M_i \to M^\omega$ together with a faithful normal conditional expectation $\rE_i : M^\omega \to \pi_i(M_i)$. Put $\psi_i = \varphi_i \circ \pi_i^{-1} \circ \rE_i$. Let $\theta \in M_\ast$ be a faithful normal state such that $(M^\theta)' \cap M = \C 1$. Since $M$ is a type ${\rm III_1}$ factor, $M^\omega$ has strictly homogeneous state space by \cite[Theorem 4.20]{AH12} and hence there exists $u_i \in \mathcal U(M^\omega)$ such that $\theta^\omega = \psi_i \circ \Ad(u_i)$. Observe that $\sigma_t^{\theta^\omega} = \Ad(u_i^*) \circ \sigma_t^{\psi_i} \circ \Ad(u_i)$ for all $t \in \R$. Put $P_i = u_i^* \pi_i(M_i) u_i$ and observe that $P_i$ is globally invariant under the modular automorphism group $(\sigma_t^{\theta^\omega})$.

By Theorem \ref{thmA}, there exists a unitary $v \in \mathcal U((M^\theta)^\omega)$ such that $P_1$ and $vP_2v^*$ are $\ast$-free inside $M^\omega$ with respect to the ultraproduct state $\theta^\omega$. Observe that since $P_i$ is globally invariant under the modular automorphism group $(\sigma_t^{\theta^\omega})$ for all $i \in \{1, 2\}$ and since $v \in \mathcal U((M^\omega)^{\theta^\omega})$, we have that $P = P_1 \vee v P_2 v^*$ is globally invariant under the modular automorphism group $(\sigma_t^{\theta^\omega})$ and hence by \cite[Theorem IX.4.2]{Ta03}, there exists a faithful normal conditional expectation $\rE : M^\omega \to P$. Moreover, by uniqueness of the free product von Neumann algebra and since $v \in \mathcal U((M^\omega)^{\theta^\omega})$, we have the following state-preserving $\ast$-isomorphisms
$$(P, \theta^\omega|_P) \cong (P_1, \theta^\omega |_{P_1}) \ast (v P_2 v^*, \theta^\omega |_{v P_2 v^*}) \cong (M_1, \varphi_1) \ast (M_2, \varphi_2).$$ 
Therefore, the free product von Neumann algebra $(M_1, \varphi_1) \ast (M_2, \varphi_2)$ embeds with expectation into $M^\omega$. This finishes the proof of Corollary \ref{corB}.

\subsection*{Proof of Theorem \ref{thmC}}
We follow the same strategy as in the proof of \cite[Corollary 2.4]{Po95b}. Since any element in $M^\omega$ can be approximated in the uniform norm $\|\cdot\|_\infty$ by finite linear combinations of projections in $M^\omega$, we only have to prove Theorem \ref{thmC} when $x=\sum_{j=1}^k\alpha_jf_j$, where $\alpha_j\in\C$ and $f_j \in M^\omega$ is a projection for all $1 \leq j \leq k$. Let $B \subset M^\omega$ be the von Neumann subalgebra generated by the set $\{\sigma_t^{\varphi^\omega}(f_j) : 1 \leq j \leq k, t \in \R\}$. Observe that $B$ has separable predual and is globally invariant under the modular automorphism group $(\sigma_t^{\varphi^\omega})$.

\begin{claim}
There exists a diffuse von Neumann subalgebra $A\subset Q^\omega$ with separable predual such that 
$$\|p(f-\varphi^\omega(f)1)p\|_\infty=\frac{1}{n}-2\frac{\varphi^\omega(f)}{n}+ \sqrt{ 4 \, \varphi^\omega(f)(1-\varphi^\omega(f))\frac{1}{n}(1-\frac{1}{n}) }$$
for all nonzero projections $p\in A$ satisfying $\varphi^\omega(p)=1/n$ and and all nonzero projections $f\in B$ satisfying $\varphi^\omega(f)\leq 1-1/n$.
\end{claim}
\begin{proof}[Proof of the Claim]
Let $A \subset Q^\omega$ be any diffuse von Neumann subalgebra with separable predual. By Theorem \ref{thmA}, up to conjugating $A$ by a unitary element $v \in \mathcal U(Q^\omega)$, we may assume that $A$ and $B$ are $\ast$-free inside $M^\omega$ with respect to the ultraproduct state $\varphi^\omega$. Let $p\in A$ and $f\in B$ as in the statement. Then the von Neumann subalgebra  $N \subset M^\omega$ generated by $p$ and $f$ is $\ast$-isomorphic in a state-preserving way to the free product von Neumann algebra 
$$(N, \varphi^\omega |_N) \cong (\C p \oplus \C p^\perp, \varphi^\omega |_{\C p \oplus \C p^\perp}) \ast (\C f \oplus \C f^\perp, \varphi^\omega |_{\C f \oplus \C f^\perp}).$$
The computation in \cite[Example 2.8]{Vo86} gives the desired formula for the operator norm of $p(f-\varphi^\omega(f)1)p \in N$.
\end{proof}

From now on, we fix such a diffuse von Neumann subalgebra $A \subset Q^\omega$ as in the Claim. Let $n \geq 1$ and $\{p_i\}_{i=1}^n$ be any partition of unity with projections in $A$ having the same trace in $Q^\omega$. Define the unitary element $u=\sum_{i=1}^n\lambda^ip_i \in \mathcal U(A)$ where $\lambda=\mathrm{exp}(2\pi{\rm i}/n)$. One can easily check that $\sum_{i=1}^np_iyp_i=\frac{1}{n}\sum_{\ell=1}^nu^\ell y u^{-\ell}$ for all $y\in M^\omega$. Recall that $x=\sum_{j=1}^k\alpha_jf_j$. Then we have 
\begin{align*}
\left\|\varphi^\omega(x)1-\frac{1}{n}\sum_{\ell=1}^nu^\ell x u^{-\ell}\right\|_\infty
&=\left\|\sum_{j=1}^k\alpha_j\varphi^\omega(f_j)1-\sum_{i=1}^np_i \left(\sum_{j=1}^k\alpha_jf_j\right)p_i\right\|_\infty\\
&\leq\sum_{j=1}^k|\alpha_j|\left\|\varphi^\omega(f_j)1-\sum_{i=1}^np_if_jp_i\right\|_\infty\\
&=\sum_{j=1}^k|\alpha_j| \left\|\sum_{i=1}^np_i(\varphi^\omega(f_j)1-f_j)p_i\right\|_\infty\\
&=\sum_{j=1}^k|\alpha_j|\max \left\{\|p_i(\varphi^\omega(f_j)1-f_j)p_i\|_\infty : 1 \leq i \leq k \right\}.
\end{align*}
When $n \to \infty$, we have $\varphi^\omega(f_j)\leq1-1/n$ for all $1 \leq j \leq k$ and hence the Claim implies that 
$$\lim_{n \to \infty}\left\|\varphi^\omega(x)1-\frac{1}{n}\sum_{\ell=1}^nu^\ell x u^{-\ell}\right\|_\infty = 0.$$
This shows that $\varphi^\omega(x)1 \in \overline{{\rm co}}^{\|\cdot\|_\infty}\{uxu^* : u \in \mathcal U(Q^\omega)\}$. Since $Q^\omega \subset (M^\omega)^{\varphi^\omega}$, if we apply $\varphi^\omega$ to any element in $\overline{{\rm co}}^{\|\cdot\|_\infty}\{uxu^* : u \in \mathcal U(Q^\omega)\} \cap \C1$, we finally obtain 
$$\overline{{\rm co}}^{\|\cdot\|_\infty}\{uxu^* : u \in \mathcal U(Q^\omega)\} \cap \C1 = \{\varphi^\omega(x) 1\}.$$
This finishes the proof of Theorem \ref{thmC}.

\end{document}